\documentclass{amsart}
\usepackage[utf8]{inputenc}
\usepackage{amsmath,amssymb,color}
\usepackage{mathtools}
 \usepackage{hyperref}
  \hypersetup{
colorlinks=true,
linkcolor=black,
citecolor=black
}
\usepackage{lineno}
\usepackage[normalem]{ulem}

\numberwithin{equation}{section}

\theoremstyle{plain}

\newtheorem{thm}[equation]{Theorem}

\newtheorem{ass}[equation]{Assumption}

\newtheorem{lem}[equation]{Lemma}

\newtheorem{fact}[equation]{Fact}

\newtheorem{qu}[equation]{Question}
\newtheorem{claim}{Claim}

\newtheorem*{claim*}{Claim}

\theoremstyle{definition}
\newtheorem{defn}[equation]{Definition}

\newtheorem{defnnotn}[equation]{Definition-Notation}

\newtheorem{rmk}[equation]{Remark}

\newcommand{\pcoor}[1]{%
  \begingroup\lccode`~=`: \lowercase{\endgroup
  \edef~}{\mathbin{\mathchar\the\mathcode`:}\nobreak}%
  \begingroup
  \mathcode`:=\string"8000
  #1%
  \endgroup 
}

\newcommand{\rainbow}{\mathrel{\smash{\raise-.1ex\hbox to 0pt{$\kern1pt\smallfrown$\hss}\raise.3ex\hbox{$\frown$}}}}

\newcommand{\PP}{\mathbb{P}}
\newcommand{\NN}{\mathbb{N}}

\newcommand{\ZZ}{\mathbb{Z}}

\newcommand{\QQ}{\mathbb{Q}}

\newcommand{\CC}{\mathbb{C}}

\newcommand{\Disc}{\mathrm{Disc}}

\def\x{\tilde x}
\def\y{\tilde y}
\def\z{\tilde z}

\def\p[#1:#2]{[#1\,{:}\,#2]}

\title[Characterizing Quasi-Homogeneity]{A Characterization of Quasi-homogeneous bivariate Polynomials}

\author[D. Bradley-Williams]{David Bradley-Williams}
\address{David Bradley-Williams, Institute of Mathematics, Czech Academy of Sciences, \v Zitn\'a 25, 115 67 Praha 1, Czech Republic.}
\email{williams@math.cas.cz}
\author[P. Cubides Kovacsics]{Pablo Cubides Kovacsics}
\address{Pablo Cubides Kovacsics, Departamento de Matemáticas, Universidad de los Andes,  
Carrera 1 no. 18A - 12,
Edificio H, Bogotá, Colombia.}
\email{p.cubideskovacsics@uniandes.edu.co}
\author[I. Halupczok]{Immanuel Halupczok}
\address{Immanuel Halupczok, Heinrich-Heine-Universit\"at D\"usseldorf, Mathe\-matisch-Natur\-wissen\-schaft\-liche Fakult\"at,
Universit\"atsstr.\ 1, 40225 D\"usseldorf, Germany}
\email{math@karimmi.de}
\thanks{This research was partially supported by the research training group \emph{GRK 2240: Algebro-Geometric Methods in Algebra, Arithmetic and Topology}, funded by the DFG. D.B-W. supported by the Czech Academy of Sciences CAS (RVO 67985840) and by project EXPRO 20-31529X of the Czech Science Foundation (GA\v CR)}

\begin{document}
\begin{abstract}
If a reduced bivariate polynomial is quasi-homogeneous, then its discriminant is a monomial. Over fields of characteristic $0$, we show that if one adds another simple condition, this becomes an equivalence. We also give a third equivalent condition that is stated geometrically.
\end{abstract}



\maketitle

\section{Introduction}

Recall that a bivariate polynomial $f \in \CC[x,y]$ is called \emph{quasi-homogeneous} if the nodes of its Newton polytope lie on a line, i.e.,
if there are integers $w, \alpha, \beta$ with $\alpha, \beta$ not both equal to $0$ such that every non-zero monomial $c_{ij} x^{i}y^{j}$ of $f$ satisfies 
\begin{equation*}
w = \alpha i+\beta j. 
\end{equation*}
One calls $\alpha$ and $\beta$ the \emph{weights} of $x$ and $y$, respectively, and $(w;\alpha,\beta)$ the type of $f$. Note that we do admit negative weights, so, for example, the polynomial $1+xy+x^2y^2$ is quasi-homogeneous of type $(0;1,-1)$.

The purpose of this note is to provide, for reduced bivariate polynomials $f$,
a characterization of quasi-homogeneity which is local geometric, in the sense that it suffices to verify given conditions at every point of the Zariski closure in $\PP^1 \times \PP^1$ of the variety defined by $f$. Here and also in the remainder of this note, $\PP^1= \PP^1(\CC) = \CC \cup \{\infty\}$ is the one-dimensional projective space over $\CC$,
and we more generally identify varieties with their $\CC$-valued points. (Varieties do not need to be irreducible.)

Our main result also states yet another equivalent condition about $f$ in terms of the discriminant of $f$ with respect to one of the variables; this condition is an explicit algebraic reformulation of the geometric characterization.

Before stating the main result, let us recall the definition of discriminant (see e.g. \cite[Chapter 12]{gel-kapra-zelev}).

\begin{defn}
Let $f(x,y) \in \CC[x,y] \setminus \CC[x]$ be a polynomial. Set $n \coloneqq \deg_y f$ and let $f_n  \in \CC[x]$ be the $y^n$-coefficient of $f$, considering the latter as a polynomial in $y$. The \emph{discriminant} of $f$ with respect to $y$ is the polynomial $\Disc_y(f)\in \CC[x]$ such that for every $a \in \CC$,
\[
\Disc_y(f)(a) \coloneqq f_n^{2n-2}(a)\prod_{i < j} (b_i - b_j)^2,
\]
where the $b_i$ are the roots of $f(a, y)$.
(For $n = 1$, one sets $\Disc_y(f) = 1$.)
\end{defn}

If $V \subset \CC^2$ is the variety defined by a polynomial  $f = \sum a_{ij}x^iy^j \in \CC[x,y]$, then its Zariski closure
in $\PP^1 \times \PP^1$ is defined by $\hat f(x,\x,y,\y) = \sum a_{ij}x^i \x^{m-i}y^j\y^{n-j}$. We call $\hat f$ the \emph{multi-homogenization of $f$} (see later Definition-Notation \ref{defnota:homogenization}).

Here is the precise formulation of our main result.

\begin{thm}\label{thm:main} 
Let $f(x,y)\in \CC[x,y]$ be a complex bivariate polynomial written as $\sum_{i=0}^n f_i(x)y^i$ where $f_i \in \CC[x]$, and $f_0, f_n$ are nonzero polynomials. Suppose that $f$ is reduced, i.e., no irreducible factor appears multiple times. We write $\hat f \in \CC[x, \tilde x, y, \tilde y]$ for the multi-homogenization of $f$. Then the following are equivalent:
 \begin{enumerate}
     \item[\hypertarget{A}{(A)}] $f$ is quasi-homogeneous where the weight of $x$ is non-zero; 
     \item[\hypertarget{B}{(B)}]  $f_0$, $f_n$ and $\Disc_y(f)$ are (non-zero) monomials (in $x$);
     \item[\hypertarget{C}{(C)}] the sub-varieties of $\PP^1 \times \PP^1$ defined by $\hat{f}$ and $y\hat{f}_y$ have no common point within $\CC^\times \times \PP^1$ (where $\hat f_y$ denotes the derivative of $\hat f$ with respect to the variable $y$).
 \end{enumerate}
\end{thm}

Note that our result is not symmetric in the variables $x$ and $y$, and to be more precise, what we characterize is quasi-homogeneity with the additional condition that the weight $\alpha$ of $x$ is non-zero. To obtain a full characterization, one could combine our condition with the variant interchanging $x$ and $y$, while the case $\alpha = 0$ is anyway simple to characterize.

Implications $\hyperlink{A}{\text{(A)}}\Rightarrow \hyperlink{B}{\text{(B)}}$ and $\hyperlink{B}{\text{(B)}} \Leftrightarrow \hyperlink{C}{\text{(C)}}$ are easy to show. (The proofs of $\hyperlink{A}{\text{(A)}}\Rightarrow \hyperlink{B}{\text{(B)}}\Rightarrow \hyperlink{C}{\text{(C)}}$ are given at the beginning of Section~\ref{s.proof}.)
What is surprising is $\hyperlink{C}{\text{(C)}}\Rightarrow \hyperlink{A}{\text{(A)}}$ (respectively $\hyperlink{B}{\text{(B)}}\Rightarrow \hyperlink{A}{\text{(A)}}$), which is our main result.

While we have stated and first prove the result over $\CC$, the conditions permit an easy generalization to
polynomials over arbitrary fields of characteristic $0$. The argument is given in
Section~\ref{sec:final}, where we also show how to deduce a similar result about geometrically reduced polynomials over fields of sufficiently big positive characteristic (see Theorem \ref{thm:main2}).

Our result is somewhat related to Bernstein-Kouchnirenko's Theorem \cite[Theorem~1.18]{Kouchnirenko1976}, which expresses the number of common zeros in $(\CC^\times)^2$ of two polynomials $f,g$ under a certain genericity condition in terms of the mixed volume of their Newton polytopes. Indeed, suppose that \hyperlink{C}{\text{(C)}} holds, so that $\hat{f}$ and $y\hat{f}_y$ have no common zero in $\CC^\times\times\PP^1$. This implies that $f$ and $g \coloneqq yf_y$ have no common zero in $(\CC^\times)^2$. Provided $f$ and $yf_y$ satisfy the genericity condition, Bernstein-Kouchnirenko's theorem implies that the mixed volume of the Newton polytopes of $f$ and $yf_y$ is 0. Therefore, the nodes of the Newton polytope of $f$ must lie on a line, that is, $f$ is quasi-homogeneous. While one can use this approach to prove Theorem~\ref{thm:main} under an additional genericity assumption on $f$, for general $f$ 
the polynomials $f$ and $yf_y$ need not satisfy the genericity condition needed by Bernstein-Kouchnirenko's theorem. The main point of our result is that the conclusion nevertheless holds without further assumptions on $f$ and $yf_y$, whereas the genericity assumption in Bernstein-Kouchnirenko's theorem cannot simply be removed.

The idea of the proof of $\hyperlink{C}{\text{(C)}}\Rightarrow \hyperlink{A}{\text{(A)}}$ is the following:
Let $X \subset \PP^1 \times \PP^1$ be the Zariski closure of the variety defined by $f$, and on $X$, consider the function $h\colon X \to \PP^1, (x,y) \mapsto \frac{y f_y(x,y)}{xf_x(x,y)}$. (It might not be well-defined everywhere, but for the sake of this sketch, let us pretend it is.) Using a well-known criterion for quasi-homogeneity (see Lemma~\ref{lem:wh1}), one finds that if $f$ is not quasi-homogeneous, then $h$ is not constant on $X$. Since $X$ is projective, this implies that there exists $(a,b) \in X$ with $h(a,b) = 0$.
It turns out that $h(a,b)$ cannot be zero if $a \in \{0,\infty\}$, so $a \in \CC^\times$, and from $h(a,b) = 0$, we deduce that $(a,b)$ is a root of $y f_y(x,y)$, contradicting $\hyperlink{C}{\text{(C)}}$.
To see that $h(0,b) \ne 0$ (the case $h(\infty, b)$ is similar), we express the branches of $X$ near $x = 0$ as Puiseux series: Assuming $y = \sum_{r \in \QQ} b_rx^r$, one verifies that the limit $\lim_{x\to 0} h(x,y)$ essentially only depends on the minimal $r$ with $b_r \ne 0$, and one in particular obtains that the limit is never $0$.

\section{Auxiliary results}

We start with the following simple but useful feature of the discriminant followed by two lemmas on quasi-homogeneity (for Lemma \ref{lem:wh1} see also e.g.\ \cite[Exercise 3 on p. 37]{kunz}). 

\begin{fact}\label{fact}
For every $a \in \CC$, $\Disc_y(f)(a)=0$ if and only if $f(a,y)$ and $f_y(a,y)$ have a common zero
or $\deg_y f(a,y) < \deg_y f(x,y)$.\qed
\end{fact}

\begin{lem}\label{lem:wh1} 
Let $f \in \CC[x,y]$ be a polynomial. Then $f$ is quasi-homogeneous if and only if there are $w,\alpha, \beta\in \CC$, not all zero, such that 
\[
wf = \alpha xf_x + \beta yf_y.
\]  
\end{lem}

\begin{proof}
Write $f$ as $\sum c_{ij} x^{i}y^{j}$ and let $I$ be the support of $f$, that is, $I=\{(i,j)\in\NN : c_{ij}\neq 0\}$. Then 
\begin{equation}\label{eq:parcials}
xf_x = \sum_{(i,j)\in I} ic_{ij}x^iy^j \text{ and } yf_y = \sum_{(i,j)\in I} jc_{ij}x^iy^j.
\end{equation}

If $f$ is quasi-homogeneous of type $(w;\alpha,\beta)$, then using that we have $w = \alpha i + \beta j$ for $(i,j)\in I$, one easily deduces that 
$wf = \alpha xf_x + \beta yf_y$.

For the converse, suppose there are $w, \alpha, \beta\in \CC$, not all zero, such that $wf=\alpha xf_x+\beta yf_y$. 
Then, by using \eqref{eq:parcials} and comparing monomials, for every $(i,j) \in I$ we obtain $\alpha i + \beta j = w$, i.e., the nodes of the Newton polytope of $f$ lie on a line, which means that $f$ is quasi-homogeneous.
Note that $w, \alpha, \beta$ can be taken to be integers, since all $(i,j) \in I$ have integer coordinates.
\end{proof}

\begin{lem}\label{lem:laurent}
Let $f \in \CC[x,y]$ be a quasi-homogeneous polynomial of type $(w; \alpha, \beta)$ with $\alpha > 0$ and $\alpha,\beta$ co-prime. Then, there are integers $k,k',\ell,d\geqslant 0$ and $c, a_1,\ldots, a_d\in \CC^\times$ such that $f$,
considered as a Laurent polynomial, can be written as
\[
f = cx^{k} y^{\ell}\prod_{i=1}^d(a_i-x^{-\beta} y^\alpha) =
cx^{k'} y^{\ell}\prod_{i=1}^d(a_ix^\beta -y^\alpha).
\]
\end{lem}
Note that at least one of those two expressions is a product of polynomials (depending on the sign of $\beta$).

\begin{proof}
Choose an enumeration of the monomials in $f$ such that $f = \sum_{i=0}^n b_ix^{k_i} y^{\ell_i}$ with $n\geqslant 0$, $b_i \in \CC^\times$ and $0\leqslant \ell_0<\ldots<\ell_n$. Since $f$ is quasi-homogeneous of type $(w;\alpha,\beta)$ (and using that $\alpha$ and $\beta$ are co-prime), we have 
\[
f = \sum_{i=0}^n b_ix^{k_0-\beta m_i} y^{\ell_0+\alpha m_i}
\]
for some $m_i\in \NN$. (That $m_i$ is non-negative follows from the assumption that $\alpha>0$ and that $\ell_0\leqslant \ell_i$ for all $0\leqslant i\leqslant m$.)
This can be written as
\[
f = x^{k_0}y^{\ell_0}\sum_{i=0}^n b_i(x^{-\beta} y^{\alpha})^{m_i}
= x^{k_0}y^{\ell_0}\cdot g(x^{-\beta} y^{\alpha})
\]
for some polynomial $g \in \CC[z]$ whose constant coefficient (which is equal to $b_0$) is non-zero, so we find $d\in \NN$ and $c,a_1,\ldots, a_d\in \CC^\times$ such that 
\[
f = cx^{k_0}y^{\ell_0}\prod_{i=1}^d (a_i - x^{-\beta} y^{\alpha}),
\]
establishing the first expression for $f$. For the second one, we pull out $x^{-\beta}$ from each factor of the product to obtain
\[
f = cx^{k_0-d\beta}y^{\ell_0}\prod_{i=1}^d (a_ix^{\beta} -  y^{\alpha}),
\]
so it remains to verify that $k' := k_0-d\beta$ is non-negative. Indeed,
this expression has a monomial of the form $cx^{k'}y^{\ell_0}\cdot (-y^{\alpha})^d$,
so we must have $k' \geqslant 0$ since no negative power of $x$ appears in $f$.
\end{proof}

\begin{defnnotn}\label{defnota:homogenization}\
\begin{enumerate}
\item By a multi-homogeneous polynomial of multi-degree $(m_1, \ldots, m_n)$ we mean a polynomial $f \in \CC[x_1, \x_1, \ldots, x_n, \x_n]$ such that every monomial of $f$ has the form $ax_1^{i_1} \x_1^{m_1-i_1}\cdots x_n^{i_n} \x_n^{m_n-i_n}$.

\item Given a polynomial $f = \sum a_{i_1\cdots i_n}x_1^{i_1}\cdots x_n^{i_n} \in \CC[x_1,\ldots, x_n]$ of degree $m_i$ in $x_i$, we define its \emph{multi-homogenization} $\hat f \in \CC[x_1, \x_1, \ldots, x_n, \x_n]$ as
\[
\hat f \coloneqq \sum a_{i_1\cdots i_n}x_1^{i_1} \x_1^{m_1-i_1}\cdots x_n^{i_n} \x_n^{m_n-i_n}.
\]
\end{enumerate}
\end{defnnotn}

Note that any multi-homogeneous polynomial $g \in \CC[x_1, \x_1, \ldots, x_n, \x_n]$ defines a sub-variety of $(\PP^1)^n$. As mentioned before (in the case $n=2$), if $\hat f$ is the multi-homogenization of a polynomial $f \in \CC[x_1, \ldots, x_n]$, the variety defined by $\hat{f}$ corresponds to the Zariski closure in $(\PP^1)^n$ of the subvariety of $\CC^n$ defined by $f$ (via the natural embedding of $\CC^n$ into $(\PP^1)^n)$.

\begin{rmk}\label{rem:homog}
If $\hat f, \hat g, \hat h$ are the multi-homogenizations of polynomials $f,g,h\in \CC[x_1,\ldots, x_n]$, then we have $f=gh$ if and only $\hat{f}=\hat{g}\hat{h}$. In particular, $f$ is irreducible if and only if $\hat{f}$ is irreducible. 
\end{rmk}

For the following lemmas, we use the following assumptions and notation (which will be relevant for the proof of $\hyperlink{C}{\text{(C)}}\Rightarrow \hyperlink{A}{\text{(A)}}$):
\begin{ass}\label{ass.f}
We fix the following objects.
\begin{itemize}
 \item $\hat f  \in \CC[x, \x, y, \y]$ is a multi-homogeneous irreducible polynomial which is not a monomial (i.e., not equal to any of $x$, $\x$, $y$, $\y$);
 \item $X \subset \PP^1 \times \PP^1$ is the irreducible projective variety defined by $\hat f$;
 \item $X_0 \subset X$ is the Zariski locally closed set given by 
\[
X_0 = \{(\p[x:\x],\p[y:\y])\in X : x\hat{f}_x(x,\x,y,\y)\neq 0\}.
\]
 \item $V\subset (\PP^1)^3$ is the projective variety defined by the multi-homogeneous polynomials (in the variables $x, \x, y, \y, z, \z$)
\[
\hat f(x,\x,y,\y) \quad \text{and}\quad  \z y\hat f_y(x,\x,y,\y) - z x\hat f_x(x,\x,y,\y). 
\]
 \item
 $h\colon X_0 \to \PP^1$ is the function sending each $(\p[x:\x],\p[y:\y])\in X_0$ to the unique $\p[z:\z]\in \PP^1$ such that $(\p[x:\x],\p[y:\y],\p[z:\z]) \in V$. More specifically:
\[
h(\p[x:\x],\p[y:\y]) = \frac{y\hat{f}_y(x,\x,y,\y)}{x\hat{f}_x(x,\x,y,\y)} \in \CC \subset \PP^1.
\]
\item $V' \subset V$ is the Zariski closure of the graph of $h$.
\end{itemize}
\end{ass}

\begin{rmk}\label{rmk.swapx}
Note that those assumptions have the following symmetry: If we set $\hat f^\#(x, \x, y, \y) \coloneqq \hat f(\x, x, y, \y)$ and let $V'^\# \subset (\PP^1)^3$ be obtained using $\hat f^\#$ instead of $\hat f$, then $(\p[x:\x],\p[y:\y],\p[z:\z]) \in V'^\#$ if and only if  $(\p[x:\x],\p[y:\y],\p[-z:\z]) \in V'$.
To see this, it suffices to verify that 
$X_0$, $V$ and $h$ (and hence also $V'$) do not change if we replace $x\hat{f}_x(x,\x,y,\y)$ by 
$-\x\hat{f}_{\x}(x,\x,y,\y)$ in the definitions of $X_0$, $V$ and $h$. Indeed, it is clear that $X_0$ does not change; to see that $V$ and $h$ do not change either, write $\hat{f} = \sum c_{ij}x^i\x^{n-i}y^j\y^{m-j}$. Then

\begin{align*}
&x\hat{f}_{x}=\sum_{i,j} ic_{ij}x^i\x^{n-i}y^j\y^{m-j} \\
&\x\hat{f}_{\x}=\sum_{i,j} (n-i)c_{ij}x^i\x^{n-i}y^j\y^{m-j}.
\end{align*}

This implies that $x\hat{f}_x + \x\hat{f}_{\x} = n \hat{f}$. Therefore, for any $(\p[x:\x],\p[y:\y])\in X$, we have that $x\hat{f}_x=-\x\hat{f}_{\x}$.

\end{rmk}

\begin{rmk}\label{rmk.swapy}
Remark~\ref{rmk.swapx} holds analogously if one swaps $y$ and $\y$ instead of $x$ and $\x$ (and again changes the sign of $z$).
\end{rmk}

In the following, we write $Y^\mathrm{Zar}$ for the Zariski closure of a set $Y \subset (\PP^1)^n$.

\begin{lem}\label{lem:zarclo1}
(Under Assumption~\ref{ass.f}.) We have $X_0^\mathrm{Zar}=X$ and  $\pi_{12}(V')=X$, where $\pi_{12}\colon (\PP^1)^3 \to (\PP^1)^2$ is the projection to the first two coordinates.
\end{lem}
\begin{proof}
Set
\[
Y \coloneqq \{(\p[x:\x],\p[y:\y]) \in (\PP^1)^2 \mid x\hat{f}_x(x,\x,y,\y)=0\}.
\]
Since $\dim Y = \dim X = 1$ and $X$ is irreducible, in order to conclude $X_0^\mathrm{Zar}=X$, it suffices to show that $X$ is not contained in $Y$. If $X$ is contained in $Y$, this implies that $\hat{f}$ divides $x\hat{f}_x$. For degree reasons, this would mean equality up to a factor from $\CC^\times$, which contradicts the irreducibility of $\hat{f}$. 

For the second part, note that 
\[\pi_{12}(V')=\pi_{12}(\operatorname{graph}(h)^\mathrm{Zar})\overset{(\star)}{=} \pi_{12}(\operatorname{graph}(h))^\mathrm{Zar}=X_0^\mathrm{Zar}=X,
\] 
where the inclusion ``$\supset$'' in $(\star)$ uses that $\pi_{12}$ is proper.
\end{proof}

Note that above any point of $X$, there are only finitely many points of $V'$.

\begin{lem}\label{lem:V'1}
(Under Assumption~\ref{ass.f}.)
$V'$ is disjoint from $\{(\p[0:1], \p[1:0])\}\times\PP^1\times\{\p[0:1]\}$.
\end{lem}

\begin{proof}
By Remarks~\ref{rmk.swapx} and \ref{rmk.swapy},
it suffices to prove that $V'$ is disjoint from $\{\p[0:1]\} \times \CC\times  \{\p[0:1]\}$.
Since this last set is a subset of $\CC^3 \subset (\PP^1)^3$, we can (for simplicity) dehomogenize everything:
Setting $f(x,y) \coloneqq \hat{f}(x, 1, y, 1) \in \CC[x,y]$,
we consider the restriction of $h$ to $X_0 \cap \CC^2$, which is given by $h(x,y) = yf_y(x,y)/(xf_x(x,y))$, and what we need to show is that the Zariski closure of its graph is disjoint from $\{0\} \times \CC \times \{0\}$.

We first treat the point $(0,0,0)$. 
Afterwards, we will reduce the general case to this one.
\medskip

\emph{Part 1: Proving that $V'$ does not contain $(0,0,0)$.}

By (a version of) Puiseux's theorem \cite[Corollary 1.5.5]{casas-alvero_2000}, we can write $f$ as
\begin{equation}\label{eq:puiseux}
f=ux^r\prod_{i=1}^k (y-s_i),
\end{equation}
where $r\in\NN$,  $u\in \CC[\![x,y]\!]$ is an invertible power series in $x,y$ and each $s_i$ is a Puiseux series in $x$, i.e., $s_i\in x^{1/N}\CC[\![x^{1/N}]\!]$ for some $N \ge 1$. (Following the convention of \cite{casas-alvero_2000}, in a Puiseux series, we allow only strictly positive powers of $x$.) Without loss of generality, replacing $x$ by $t^N$ for some suitable large integer $N$, we may suppose that all exponents in the series are integers, and therefore, we can work with power series. Indeed, note that by setting $f^\#(t,y)\coloneqq f(t^N, y)$, we obtain that the corresponding map
$h^\#(t,y)=yf^\#_y/(tf^\#_t)$ satisfies $h^\#(t,y)=Nh(t^N,y)$. Therefore, the corresponding set $V'^\#$ contains $(0,0,0)$ if and only if $V'$ does.

Next, note that in \eqref{eq:puiseux}, we have $r=0$. Indeed, set $q=u\prod_{i=1}^k (y-s_i) \in \CC[\![x, y]\!] \subset \CC(\!(y)\!)(\!(x)\!)$ and let $v_x$ denote the $x$-adic valuation on $\CC(\!(y)\!)(\!(x)\!)$. Then, we have $v_x(q)=0$ (since $u$ is invertible and $v_x(y-s_i)=0$). On the other hand, since $qx^r=f\in \CC[x,y]$, we have $q\in \CC[x,x^{-1},y]$. In particular, we can write $q$ as $\sum_{i\in I} a_ix^i$ with $a\in\CC[y]$ and $I$ a finite subset of $\ZZ$. But since $v_x(q) = 0$, we must have $a_i=0$ for all $i<0$. Therefore $q\in \CC[x,y]$. If $r>0$, then $f = x^r q$ would not be irreducible, hence $r=0$. 

Since $u$ is invertible, there is an open neighborhood $U \subset \CC^2$ of $(0,0)$ where $u$ does not vanish. Hence $X\cap U$ is the union of the graphs $\{(x,y)\in U \mid y=s_i(x)\}$ of the power series $s_i$. Note that for each of those power series, we have
\begin{equation}\label{eq.-s'}
-s'_i(x)= \frac{f_x(x,s_i(x))}{f_y(x, s_i(x))}
\end{equation}
(where $s'_i$ denotes the derivative of $s_i$). Indeed, composing the map $(\mathrm{id}, s_i)\colon \CC\to \CC^2, x\mapsto (x,s_i(x))$ with $f$ gives the zero map (because $f$ is zero on the graph of $s_i$). Thus, the derivative of the composed function is zero; expressed using the chain rule, that derivative is equal to $f_x(x,s_i(x)) + f_y(x,s_i(x))\cdot s_i'(x)$, so we obtain \eqref{eq.-s'}.

Write $s_i(x)=\sum_{i\geqslant M} b_ix^i$ for $M\geqslant 1$ and $b_M\neq 0$. Then $s_i'(x)=\sum_{i\geqslant M} ib_ix^{i-1}$ and hence
\begin{align*}
\lim_{x\to 0} h(x,s_i(x)) & = \lim_{x\to 0}  \frac{s_i(x)f_y(x,s_i(x))}{xf_x(x,s_i(x))} \\ 
& \overset{\eqref{eq.-s'}}{=} \lim_{x\to 0}  -\frac{s_i(x)}{xs_i'(x)} 
= \lim_{x\to 0} -\frac{\sum_{i\geqslant M} b_ix^i}{\sum_{i\geqslant M} ib_ix^i} 
= -\frac{1}{M}\neq 0.  
\end{align*}
Applying this to each of the $s_i$ shows that
$(0,0,0)$ does not lie in the closure of $\operatorname{graph} (h)$ in the analytic topology. The closure of the graph in the Zariski topology is the same, so $(0,0,0)\notin V'$.   

\medskip

\emph{Part 2: Proving that $V'$ does not contain 
$(0,y_0,0)$, for $y_0 \in \CC^\times$.}

Consider the change of variables $y \mapsto y+y_0$, i.e., set $f^\#(x,y)=f(x,y+y_0)$. Note that 
\begin{align*}
h^\#(x, y) & = \frac{yf^\#_y(x, y)}{xf^\#_x(x, y)} = \frac{yf_y(x, y+y_0)}{xf_x(x, y+y_0)} \\
& = \frac{y}{y+y_0}h(x, y+y_0).
\end{align*}
By Case 1, $(0,0,0)$ does not belong to the closure of the graph of $h^\#$. Therefore $(0,y_0,0)$ does not belong to the closure of the graph of $h$. 
\end{proof}

\section{Proof of Theorem \ref{thm:main}}
\label{s.proof}

Suppose $f=\sum_{i=0}^n f_i(x)y^i$ is a reduced polynomial such that $f_0$ and $f_n$ are nonzero polynomials. We show  
\[
\hyperlink{A}{\text{(A)}}\Rightarrow \hyperlink{B}{\text{(B)}}\Rightarrow \hyperlink{C}{\text{(C)}}\Rightarrow \hyperlink{A}{\text{(A)}}. 
\]

$\hyperlink{A}{\text{(A)}}\Rightarrow \hyperlink{B}{\text{(B)}}$ : Suppose $f$ is quasi-homogeneous of type $(w;\alpha,\beta)$ with $\alpha\neq 0$. Without loss we may assume that $\alpha>0$ and that $\alpha$ and $\beta$ are co-prime. It is clear that $f_0$ and $f_n$ are monomials. To see that $\Disc_y(f)$ is a monomial, by Lemma \ref{lem:laurent}, we can write $f$ both as  
\[
cx^{k} y^{\ell}\prod_{i=1}^d(a_i-x^{-\beta} y^\alpha) \text{ and }  cx^{k'} y^{\ell}\prod_{i=1}^d(a_ix^\beta -y^\alpha)
\]
where one of the two expressions is a product of polynomials, and where the $a_i$ are non-zero. Suppose the former is a product  polynomials (so $\beta\leqslant0$), the other case being similar. Since $f$ is reduced we have that $0\leqslant k, \ell\leqslant 1$ and all $a_i$ must be different. Moreover, for every $e \in \CC^\times$, the equation $a_i-e^{-\beta}y^\alpha=0$ has no multiple roots. Therefore, also $f(e,y)$ has no multiple roots. 
Since $f_n$ is a monomial, we also have $f_n(e) \ne 0$, so we obtain (by Fact~\ref{fact}) that $(\Disc_y(f))(e) \ne 0$ for all $e \in \CC^\times$. In other words, the only possible root of $\Disc_y(f)$ is $0$, meaning that it is a monomial. 

\medskip

$\hyperlink{B}{\text{(B)}}\Rightarrow \hyperlink{C}{\text{(C)}}$ : Suppose \hyperlink{B}{\text{(B)}} holds but $\hat{f}$ and $y\hat{f}_y$ have a common zero $(\p[a:1],\p[b:
\tilde{b}])\in \CC^\times\times \PP^1$. Since $f_0$ is a monomial, we must have $b\neq 0$, as otherwise, $\hat f(a,1, 0,\tilde b) = f_0(a)\tilde b^n=0$ would imply that $a=0$. Similarly, $\tilde{b}\neq 0$, since $f_n$ is a monomial. Hence, without loss, $\tilde{b}=1$. Therefore, 
\[
\hat{f}(a,1,b,1)=f(a,b)=0\quad \text{and}\quad b\hat{f}_y(a,1,b,1)=bf_y(a,b)=0.
\]
Since $b \ne 0$, the latter implies $f_y(a,b) = 0$. By Fact \ref{fact}, we obtain $\Disc_y(f)(a)=0$, which implies that $a=0$ since $\Disc_y(f)$ is a monomial, a contradiction.  

\medskip

$\hyperlink{C}{\text{(C)}}\Rightarrow \hyperlink{A}{\text{(A)}}$ : We first reduce to the case of irreducible polynomials. 

\begin{claim}\label{claim:irreducible}
It suffices to prove $\hyperlink{C}{\text{(C)}}\Rightarrow \hyperlink{A}{\text{(A)}}$ when $f$ is irreducible.     
\end{claim}

\begin{proof}
Let $f$ be a polynomial satisfying $\hyperlink{C}{\text{(C)}}$, that is, $\hat{f}$ and $y\hat{f}_y$ have no common root in $\CC^\times\times\PP^1$. Suppose further that $f=gh$ and that the implication $\hyperlink{C}{\text{(C)}}\Rightarrow\hyperlink{A}{\text{(A)}}$ holds for $g$ and $h$. We show that $f$ is quasi-homogeneous with non-zero weight of $x$. By Remark \ref{rem:homog}, we have that $\hat{f}=\hat{g}\hat{h}$. Moreover, the usual derivation rules imply 
\[
y\hat{f}_y = \hat{g}(y\hat{h}_y)+\hat{h}(y\hat{g}_y).
\]
This shows that $\hat{g}$ and $y\hat{g}_y$ (resp. $\hat{h}$ and $y\hat{h}_y$) have no common root in $\CC^\times\times\PP^1$ as otherwise $\hat{f}$ and $y\hat{f}_y$ would have one. Therefore, by assumption, both $g$ and $h$ are quasi-homogeneous with non-zero weight of $x$.
If either $g$ or $h$ is a monomial, it is easy to see that $f$ is quasi-homogeneous, so we are done.
So suppose $g$ and $h$ are not monomials. In order to deduce $\hyperlink{A}{\text{(A)}}$ for $f$, it suffices to verify that $g$ and $h$ have the same weights (up to some factor), so suppose otherwise. Using Lemma \ref{lem:laurent}, write 
\[
g=cx^{k} y^{\ell} \prod_i (a_ix^\beta - y^\alpha) \text{ and } h=dx^{k'} y^{\ell'} \prod_j (b_jx^\delta - y^\gamma)
\]
for integers $k,\ell,k',\ell', \alpha,\beta, \gamma$ and $\delta$ with $\alpha \ne 0, \gamma \ne 0$
and $c,d, a_i, b_j\in \CC^\times$, and where neither of the products over $i$ and $j$ are empty. The weight difference implies that $\beta/\alpha\neq\delta/\gamma$.

Let $a$ be any of the $a_i$ and let $b$ be any of the $b_j$. We will find a common zero $(x_0,y_0) \in \CC^\times \times \CC^\times$ of the Laurent polynomials $ax^\beta - y^\alpha$ and $bx^\delta - y^\gamma$, which is hence a common zero of $g$ and $h$. Therefore, $(\p[x_0:1],\p[y_0:1])$ is a common root of $\hat{f}$ and $y\hat{f}_y$, contradicting the assumption.

In seeking a common root of the factors above, we may suppose that $(\alpha, \gamma) = 1$ (that is, they are coprime), if necessary via a change of variables $t=y^{(\alpha,\gamma)}$.
Now let $x_0$ be any $(\delta\alpha - \beta\gamma)$th root of $\frac{a^\gamma}{b^\alpha}$.
This implies
\[
a^\gamma x_0^{\beta\gamma} = b^\alpha x_0^{\delta\alpha} \eqqcolon w.  
\]
We need to find a $y_0$ such that $y_0^\alpha = ax_0^\beta$ (which is a $\gamma$th root of $w$) and $y_0^\gamma = bx_0^\delta$ (which is an $\alpha$th root of $w$).
Let $z_0$ be a fixed $(\alpha\gamma)$th root of $w$
and let  $\zeta$ be a primitive $|\alpha\gamma|$th root of unity. Then we have 
\[
ax_0^\beta = z_0^\alpha\zeta^{i\alpha}  \text{ and } bx_0^\delta = z_0^\gamma\zeta^{j\delta} 
\]
for some integers $i$ and $j$. If we set $y_0=z_0\zeta^k$ for some integer $k$, then our two conditions on $y_0$ become
\[
z_0^\alpha\zeta^{k\alpha}=z_0^\alpha\zeta^{i\alpha} \text{ and } z_0^\gamma\zeta^{k\gamma}=z_0^\gamma\zeta^{j\delta}.  
\]
This corresponds to the modular equations 
\[
k\alpha\equiv i\alpha  \mod \alpha\gamma \quad\text{and}\quad k\gamma\equiv j\gamma \mod \alpha\gamma 
\]
which have a common solution since $\alpha$ and $\gamma$ are coprime. 
\end{proof}

To show $\hyperlink{C}{\text{(C)}}\Rightarrow\hyperlink{A}{\text{(A)}}$ we will prove its contrapositive, so assume the negation of $\hyperlink{A}{\text{(A)}}$, that is, either $f$ is
not quasi-homogeneous, or it is quasi-homogeneous only using $\alpha=0$ as the weight of $x$. The latter means that
$f$ is a polynomial in $x$ only (since we assumed $f_0 \ne 0$) but not a monomial. Then $f$ has a root $(a, 0)$ for some $a \in \CC^\times$, and since $y \hat f_y$ is the zero-polynomial,
$(\p[a:1],\p[0:1])$ is a common root of $\hat f$ and $y \hat f_y$, contradicting $\hyperlink{C}{\text{(C)}}$.

We are left with the case where $f$ is not quasi-homogeneous.
We let $\hat f$ be the multi-homogenization of $f$ and use all the notation from Assumption~\ref{ass.f}.

\begin{claim}\label{claim:znotconst1}
The function $h\colon X_0 \to \PP^1$ is not constant. 
\end{claim}
\begin{proof}
Suppose $h$ is constant. Then
\[
y\hat{f}_y(x,\x,y,\y) = x\hat{f}_x(x,\x,y,\y)\cdot c
\] 
for some constant $c \in \CC$, for all $(\p[x:\x],\p[y:\y]) \in X_0$.
Since this polynomial equality holds on $X_0$, it also holds on the Zariski closure of $X_0$, which is $X$ by Lemma \ref{lem:zarclo1}. Thus, the polynomial $y\hat{f}_y- x\hat{f}_x\cdot c \in \CC[x,\x,y,\y]$ is a multiple of $\hat{f}$, i.e., 
\[
y\hat{f}_y - x\hat{f}_x\cdot c = g\hat{f}
\]
for some $g \in \CC[x,\x,y,\y]$. For degree reasons, $g$ is constant equal to some $d \in \CC$. Setting the variables $\x=\y=1$, we obtain $yf_y - xf_x\cdot c = df$. Now Lemma \ref{lem:wh1} implies that $f$ is quasi-homogeneous, contradicting our assumption. 
\end{proof}

\begin{claim}\label{claim:surjectivity1}
Let $\pi_3\colon (\PP^1)^3\to \PP^1$ be the projection  onto the third coordinate. Then $\pi_3|_{V'}$ is surjective.
\end{claim}
\begin{proof}
Since $\pi_3$ is a proper morphism, it is a closed map. Hence the image of $V'$ is closed. But in $\PP^1$, a closed set is either finite or the whole $\PP^1$. If the image is finite, it is a singleton (by irreducibility of $V'$). But then $h$ would be constant, contradicting Claim~\ref{claim:znotconst1}.
\end{proof}

By Claim~\ref{claim:surjectivity1}, there exists a $(\p[x_0:\x_0],\p[y_0: \y_0])\in X$ such that $(\p[x_0:\x_0],\p[y_0: \y_0],\p[0:1])\in V'$. By Lemma~\ref{lem:V'1}, we have $\p[x_0:\x_0] \notin \{\p[0:1], \p[1:0]\}$. Therefore $(\p[x_0:\x_0],\p[y_0: \y_0])\in \CC^\times\times\PP^1$ is a root of $\hat{f}$ (by definition of $X$) and a root of $y\hat{f}_y$ by definition of $h$, completing the proof.

\section{Final remarks}
\label{sec:final}

We show in this section how Theorem \ref{thm:main} can be extended to arbitrary fields of characteristic 0 and, for a fixed degree $d$, over fields of characteristic $p$ for large $p$ depending on $d$. This is the content of the following  theorem. 

\begin{thm}\label{thm:main2} Let $K$ be a field of characteristic $p\geqslant 0$.
\begin{enumerate}
    \item If $p=0$, Theorem \ref{thm:main} holds over $K$. 
    \item For every $d\in\NN$ there is $N_d\in \NN$ such that if $p>N_d$, Theorem \ref{thm:main} holds over $K$ for all geometrically reduced polynomials of degree smaller or equal than $d$.  
\end{enumerate}
\end{thm}

\begin{proof}
For (1), let $K$ be a field of characteristic 0 and $f\in K[x,y]$ be a bivariate polynomial. Let $c=(c_{i,j})$ be the coefficients of $f$. Then $\QQ(c)$ is an extension of $\QQ$ of finite transcendence degree. Let $\varphi\colon \QQ(c)\to \CC$ be an embedding and $f^\varphi\in\CC[x,y]$ be the image of $f$ under $\varphi$. Since the result holds for $f^\varphi$, it is not difficult to see that the result holds for $f$. Indeed, note that all the conditions in the theorem hold for $f$ (i.e., being quasi-homogeneous, reduced, etc.) if and only if they hold for $f^\varphi$. Note also that since we are in characteristic $0$, being reduced is equivalent to being geometrically reduced. 

\medskip For (2), let us first show the statement when $K$ is algebraically closed. This follows by noting that for all polynomials of degree smaller or equal than $d$, Theorem \ref{thm:main} can be expressed by a first-order sentence in the language of rings. Thus, by the transfer principle of algebraically closed fields (see \cite[Corollary B.12.4]{transseries}), there is $N_d\in \NN$ such that the same statement holds for all algebraically closed fields of characteristic $p>N_d$. Now, to conclude for all fields of characteristic $p>N_d$, fix first an algebraically closed field $F$ of characteristic $p>N_d$ with infinite transcendence degree over $\mathbb{F}_p$. Let $K$ be any field of characteristic $p$ and $f\in K[x,y]$ be a bivariate polynomial which is geometrically reduced. Let $c=(c_{i,j})$ be the coefficients of $f$. Then $\mathbb{F}_p(c)$ is an extension of $\mathbb{F}_p$ of finite transcendence degree. Letting $\varphi\colon \mathbb{F}_p(c)\to F$ be an embedding and $f^\varphi\in F[x,y]$ be the image of $f$ under $\varphi$ we conclude as in case (1), noting that $f^\varphi$ is reduced. 
\end{proof}

We finish by asking the following questions: 

\begin{qu}
Does Theorem \ref{thm:main} hold for geometrically reduced polynomials over all fields of positive characteristic?
\end{qu}

\begin{qu}
Is there a suitable analogue of Theorem \ref{thm:main} in higher dimension (i.e., for curves in $\CC^n$ or for hypersurfaces in $\CC^n$, or maybe for arbitrary varieties in $\CC^n$)?
\end{qu}

\bibliographystyle{amsplain}
\bibliography{references}

\end{document}